\documentclass[12pt]{amsart}
\usepackage[utf8]{inputenc}
\usepackage{amssymb, amsmath}
\usepackage{amscd}
\usepackage{graphicx}
\graphicspath{ {images/} }
\usepackage{wrapfig}
\usepackage[matrix,arrow,curve]{xy}\usepackage{mathabx}

\let\oldtocsection=\tocsection
 
\let\oldtocsubsection=\tocsubsection
 
\let\oldtocsubsubsection=\tocsubsubsection
 
\renewcommand{\tocsection}[2]{\hspace{0em}\oldtocsection{#1}{#2}}
\renewcommand{\tocsubsection}[2]{\hspace{1em}\oldtocsubsection{#1}{#2}}
\renewcommand{\tocsubsubsection}[2]{\hspace{2em}\oldtocsubsubsection{#1}{#2}}

\usepackage{color}

\begin{document}
\newtheorem{predl}{Proposition}[section]
\newtheorem{thrm}[predl]{Theorem}
\newtheorem{lemma}[predl]{Lemma}
\newtheorem{cor}[predl]{Corollary}
\theoremstyle{definition}
\newtheorem{df}{Definition}
\renewcommand{\proofname}{\textnormal{\textbf{Proof:  }}}
\newcommand{\rmk}{\textnormal{\textbf{Remark:  }}}
\newcommand{\C}{\mathbb C}
\renewcommand{\refname}{Bibliography}
\newcommand{\Z}{\mathbb Z}
\newcommand{\Q}{\mathbb Q}
\renewcommand{\o}{\otimes}
\newcommand{\e}{\mathcal{P}}
\newcommand{\R}{\mathbb R}
\renewcommand{\O}{\mathcal O}
\newcommand{\g}{\mathfrak g}
\newcommand{\D}{\mathcal D}
\renewcommand{\Im}{\operatorname{Im}}
\newcommand{\Hom}{\operatorname{Hom}}
\newcommand{\Alb}{\operatorname{Alb}}
\newcommand{\Hilb}{\operatorname{Hilb}}
\newcommand{\acts}{\lefttorightarrow}
\binoppenalty = 10000
\relpenalty = 10000

\title{K\"ahler submanifolds in Iwasawa manifolds}

\author{Vasily Rogov}

\begin{abstract}

Iwasawa manifold is a compact complex  homogeneous manifold isomorphic to a quotient $G/\Lambda$, where $G$ is the group of complex unipotent $3 \times 3$ matrices and $\Lambda \subset G$ is a cocompact lattice. 
We prove that any compact complex curve in an Iwasawa manifold is contained in a holomorphic subtorus. We also prove that any complex surface in an Iwasawa manifold is either an abelian surface or a  K\"ahler non-projective isotrivial elliptic surface of Kodaira dimension one.
In the Appendix we show that any subtorus in Iwasawa manifold carries complex multiplication.

\end{abstract}

\maketitle

\tableofcontents

\section{Introduction}
\label{Intro}

 Recall that a  {\it nilmanifold} is a compact  manifold which admits a  transitive action of a nilpotent Lie group. The classical theorem of Mal'\v{c}ev \cite{Mal} states that any nilmanifold is diffeomorphic to a quotient of a connected simply connected nilpotent Lie group $G$ by a cocompact  lattice $\Lambda \subset G$. Moreover, the group $G$ can be obtained as the pro-unipotent completion of  $\Lambda$ (which is also known as {\it Mal'\v{c}ev completion}). 
  
 A {\it complex nilmanifold} is a pair $(N,J)$, where $N = G/\Lambda$ is a nilmanifold  and $J$ is a $G$-invariant integrable complex structure on $N$.
 
 The topology of nilmanifolds is well understood: any nilmanifold is diffeomorphic to an iterated tower of principal toric bundles. However, the complex geometry of nilmanifolds is much more varied and rich: for example, one is not always able to choose these toric bundles to be holomorphic. For further discussion on this and related problems see e.g. \cite{Rol}
 
 One of the reasons why  nilmanifolds provoke such an interest in complex geometry is that they deliver a number of examples of non-K\"ahler complex manifolds. Indeed, a complex nilmanifold  never admits K\"ahler metric unless it is a torus (\cite{BG}). Moreover, a nilmanifold which is not a torus is not homotopically equivalent to any K\"ahler manifold: its de Rham algebra is never formal (\cite{Has}). In \cite{BR} Bigalke and Rollenske constructed complex nilmanifolds with arbitrary long non-degenerate Fr\"olicher spectral sequence, which means that complex nilmanifolds can be in some sense arbitrary far from being K\"ahler.
 
 At this point it is rather natural to try to describe K\"ahler submanifolds in complex nilmanifolds. 
 
 We solve this problem for Iwasawa manifolds which are a special but important class of 3-dimensional complex nilmanifolds. It turns out that any complex submanifold of an Iwasawa manifold admits a K\"ahler metric.If it is a curve then it is necessary contained in a certain holomorphic torus inside an Iwasawa manifold. If it is a surface, then it is either an abelian surface or a total space of principal elliptic bundle over a curve of general type. In the latter case we show that it admits a K\"ahler metric, but can never be projective.
 
  In the Appendix we also show that any abelian surface inside an Iwasawa manifold is of maximal Picard number and hence splits as a product of two isogenous  elliptic curves with complex multiplication.
 \\
 
{\bf Acknowledgments}. I am thankful to Misha Verbitsky for posing this problem to me, plenty of fruitful discussions and help with the preparation of the manuscript. I am also thankful to Rodion D\'eev  and to the referee commitee of M\"obius Contest for  useful comments on a preliminary version of this paper.

 \section{Iwasawa manifolds} 
 \label{Iwasawa}

Starting from now,  $G$ is the group of complex unipotent  $3 \times 3$-matrices, i.e. matrices of the type 
$ \begin{pmatrix} 
1 & z_1 & z_2\\
0& 1 & z_3\\
0&0&1
\end{pmatrix} $ with $z_i \in \C$.
 
\begin{df}
An {\it Iwasawa manifold} is a  compact complex manifold which admits a transitive holomorphic action of $G$  with discrete stabilizer.
\end{df}

As  follows from the Mal'\v{c}ev theorem \cite{Mal}, any Iwasawa manifold is isomorphic to $G/\Lambda$ for  a cocompact lattice $\Lambda \subset G$. The complex structure is induced by the standard complex structure on $G$. For different choices of $\Lambda$ one gets in general pairwise non-biholomorphic (and even sometimes non-homeomorphic!) complex nilmanifolds.

The group $G$ is a central extension of a two-dimensional commutative complex Lie group by one-dimensional: $$1 \to \C \to G \to \C^2 \to 1.$$ The corresponding cocycle for  its Lie algebra $\mathfrak{g}$ can be represented by the standard complex symplectic form  $ \sigma \colon \Lambda^2\mathbb{C}^2 \to \C$.

Consider an Iwasawa manifold $I$ isomorphic to $G/\Lambda$. Let $\Lambda^{ab}$ be the abelianization of  $\Lambda$ and $\Lambda_Z: = \Lambda \bigcap Z(G)$  the intersection of $\Lambda$ with the group center $Z(G)$. The above-mentioned Mal'\v{c}ev theorem  (\cite{Mal}) also implies that both $\Lambda^{ab}$ and $\Lambda_Z$ are lattices in $G^{ab} = \C^2$ and $Z(G) = \C$ respectively. Thus $E = \C / \Lambda_Z$ is an elliptic curve and  $T = \C^2 / \Lambda^{ab}$ is a compact complex torus. The projection $p \colon G \to G/Z(G)$ descends to  holomorphic  map $\pi \colon I \to T$.  The elliptic curve $E$ acts on its fibres and it its easy so see that $\pi$ is a holomorphic principal $E$-bundle. Further on we are going to call it the {\it Iwasawa bundle}.

If we identify the group $G$ with the set of matrices $\left ( \begin{smallmatrix} 1& x& y\\ 0& 1& z\\ 0& 0& 1 \end{smallmatrix} \right )$, the forms $\alpha = dx, \beta = dz$ and $\gamma = xdy$ descend to holomorphic $1$-forms on $I$ and satisfy relations $d\alpha = d\beta = 0$ and  $d\gamma = \alpha \wedge \beta$.  

\begin{thrm}[Fern\'andez, Gray.]
For any Iwasawa manifold $I$ let $A^{\bullet}$ be the subalgebra of the de Rham algebra $\Lambda^{\bullet}I$ generated by $\alpha, \beta, \gamma$ and their complex conjugates. Then the tautological embedding of $A^{\bullet} \hookrightarrow \Lambda^{\bullet}I$ is a quasi-isomorphism.
\end{thrm}
\begin{proof}
See \cite{FG}.
\end{proof}.

The explicit computation of Betti numbers of Iwasawa manifolds can  be also found in \cite{FG}

\section{Principal elliptic bundles.}
\label{Ebundles}
\subsection{Holomorphic principal elliptic bundles}
\label{Hofer}
For the duration of this section fix a compact complex manifold $B$ and an elliptic curve $E = \C/\Gamma$.

The aim of this section is to recall some techniques of working with holomorphic principal $E$-bundles  and to apply them in the case of Iwasawa bundle.
  
 Let $\mathcal{E}_B$ denote the sheaf of holomorphic functions on $B$ with values in $E$. The isomorphism classes of holomorphic principal $E$-bundles on $B$ are in one-to-one correspondence with the elements of $H^1(B, \mathcal{E}_B)$.

Let $\mathcal{P} \colon M \to B$ be a principal holomorphic $E$-bundle. Denote by $\mathcal{V} = \operatorname{Ker} D\mathcal{P}$ the subbundle in $TM$ tangent to the fibres of $\mathcal{P}$. Choose any $\mathcal{E}_B$-equivariant Ehresmann connection $\mathcal{H}$, that is a $\mathcal{E}_B$-equivariant subbundle in $TM$,  which gives the decomposition $TM = \mathcal{V} \oplus \mathcal{H}$. It induces an affine connection on $\mathcal{P}$ with values in the Lie algebra $\mathfrak{e}$ of $E$. By the standard results of the theory of principal $G$-bundles  (see e.g.\cite{S}) the curvature form of this affine connection descends to a closed  2-form on $B$  with values in the Lie algebra of $E$ (here we also use the fact that the group $E$ is abelian) . Its cohomology class $c_1(\mathcal{P})$ (the first Chern class) is a topological invariant of $\mathcal{P}$.  

As far as $E$ is a one-dimensional complex commutative Lie group, it is natural to think about $c_1(\mathcal{P})$ as an element of $H^2(B, \C)$.

Since $\mathcal{E}_B$ is a sheaf of commutative groups, $H^1(B, \mathcal{E}_B)$ carries a structure of commutative group itself. With  respect to this structure $$c_1 \colon H^1(B, \mathcal{E}_B) \to H^2(B, \C)$$ is a homomorphism.

There is also another point of view on  $c_1(\mathcal{P})$, which was studied by T. H\"ofer in \cite{Hof}.

The short exact sequence $$0 \to \Gamma \to \C \to E \to 0$$ gives rise to exact sequence of sheaves 
\begin{equation}\label{shof}
0 \to \underline{\Gamma} \to \mathcal{O}_B \to \mathcal{E}_B \to 0. 
\end{equation}
 Here $\underline{\Gamma}$ is the subsheaf of the constant sheaf $\underline{\C}$ consisting of the sections which take values in $\Gamma$. Therefore we have the cohomology long exact sequence:
\[
\ldots \to H^1(B, \mathcal{O}_B) \to H^1(B, \mathcal{E}_B) \xrightarrow{c_1^{\mathbf z}} H^2(B, \underline{\Gamma}) \to H^2(B, \mathcal{O}_B)\to \ldots 
\]

\begin{predl}[\cite{Hof}]
Consider the first Chern class $c_1$ as a homomorphism from $H^1(B, \mathcal{E}_B)$ to $H^2(B, \C)$. Then $c_1$ coincides  with the composition of $c_1^{\mathbf z}$ and the natural map $H^2(B, \underline{\Gamma}) \to H^2(B, \C)$ induced by the embedding $\Gamma \hookrightarrow \C$.
\end{predl}

 The following proposition gives more explicit description of the group structure in $H^1(B, \mathcal{E}_B)$.
  
\begin{predl}
Let $\mathcal{P}_1$ and $\mathcal{P}_2$ be two holomorphic principal $E$-bundles over $B$ with total spaces $M_1$ and $M_2$. Let $[\mathcal{P}_i] \in H^1(B, \mathcal{E}_B)$ be the corresponding cocycles. Define $\mathcal{P}_1 \times_E \mathcal{P}_2$ as the quotient of $M_1 \times_B M_2$ over the diagonal action of $E$.
Then $$[\mathcal{P}_1 \times_E \mathcal{P}_2] = [\mathcal{P}_1] + [\mathcal{P}_2]$$
\end{predl}

\begin{proof}
This is a classical fact. To prove it one can for example use  Eckmann-Hilton argument \cite{EH}.
\end{proof}

\begin{predl}
Assume that an elliptic bundle $\mathcal{P} \colon M \to B$ admits a multisection, i.e. there exists a submanifold $Z \subset M$, such that the restriction $\mathcal{P}|_Z \colon Z \to B$ is a finite morphism. Then $[\mathcal{P}]$ is a torsion element in $H^1(B, \mathcal{E}_B)$. Particularly, $c^{\mathbf z}_1(\mathcal{P})$ is a torsion class in $H^2(B, \Lambda) = H^2(B , \Z) \o_{\Z} \Lambda$ and $c_1(\mathcal{P})$ vanishes.
\end{predl}

\begin{proof}
Let $k$ be the degree of $\mathcal{P}|_{Z} \colon Z \to B$. 

 For any $b \in B$ the intersection of $Z$ with the fibre $E_b = \mathcal{P}^{-1}(b)$ is a zero-dimensional subscheme $Z_b$ on $E_b$ of degree $k$.  So the multisection $Z$ defines a section $b \mapsto Z_b$ of the fibrewise Hilbert scheme of points $\Hilb^k(\mathcal{P})$. It can be contracted on the fibrewise $k$-th symmetric power $S^k\mathcal{P}$,  which projects to the total space of the bundle $\mathcal{P}^{\times_E k}$.  The section of $S^k \mathcal{P}$ produces a section of $\mathcal{P}^k$. 

Hence, the $k$-th power of $\mathcal{P}$ is trivial.
\end{proof}

Now we need to make a pair of additional remarks in the case when $B$ is K\"ahler. Firstly, in this case the map $$H^2(B, \underline{\Gamma}) \to H^2(B, \mathcal{O}_B)$$ equals to  the composition of natural morphism $H^2(B, \underline{\Gamma}) \to H^2(B, \C)$ and  the projection in the Hodge decomposition $$H^2(B , \C) \to H^{0,2}(B, \C) = H^2(B, \mathcal{O}_B).$$ This shows that the image of $c_1$ has vanishing $(0,2)$-part.

Secondly, the following useful theorem was proved by Blanchard \cite{Bl}:

\begin{thrm}[Blanchard]
Let $\mathcal{P} = (p  \colon M \to B)$ be a holomorphic principal elliptic bundle. Assume that $B$ is K\"ahler and $H^2(B, \Z)$ is torsion-free. Then $M$ is K\"ahler if and only if $c_1(\mathcal{P}) = 0$.
\end{thrm}

We  sketch the proof for the case when $B$ is a curve.

\begin{proof}[Sketch of proof]
Assume that $c_1(\mathcal{P}) \neq 0$ and $M$ is K\"ahler. 

It is easy to see that the pullback of $\mathcal{P}$ to its own total space admits a section, and hence $p^*(c_1(\mathcal{P})) = 0$. Since we assumed that $B$ is a curve, either $c_1(\mathcal{P})$ or $-c_1(\mathcal{P})$ can be represented by a K\"ahler form $\omega_B$. Let $\omega_M$ be a K\"ahler form on $M$. Thus $p^*\omega_B = d\eta$ for some 1-form $\eta$ on $M$.  Now we have
$$\operatorname{Vol}(M) = \int_{B} \Bigl( \int_{p^{-1}(b)} \omega_M \Bigr )\omega_B = \int_Mp^* \omega_B \wedge \omega_M = \int_M d(\eta \wedge \omega_M) = 0.$$  Clearly, this is impossible. This proves the  ''only if'' part.

To prove the sufficiency one can use the fact that the vanishing of $c_1(\mathcal{P})$ is equivalent to existence of a flat $E$-connection on $\mathcal{P}$.  On a small open set $U \subset B$ one is  able to  choose a splitting $p^{-1}(U) = U \times E$ and find there a closed positive $(1,1)$ form $\xi_0$ which is a lifting of a K\"ahler form from $E$. Averaging by the action of $E$ one can make this form to be constant among the fibres. Then, translating by the flat connection on $\mathcal{P}$, glue it into a global form $\xi$. 

If $\omega_B$ is a K\"ahler form on the base, a K\"ahler form on $M$ is given by $$\omega_M = p^*\omega_B + \xi.$$
\end{proof}

\begin{rmk}Indeed $c_1$ is the only topological invariant of a principal elliptic bundle. To see this, replace the sheaves $\mathcal{O}_B$ and $\mathcal{E}_B$ in the exact sequence  (\ref{shof}) with the sheaves of smooth  functions $\mathcal{C}^{\infty}(B, \C)$ and $\mathcal{C}^{\infty}(B, E)$. Consider the corresponding cohomology long exact sequence.  Since $\mathcal{C}^{\infty}(B, \C)$ is acyclic,  $c^{\mathbf z}$ defines a bijection between the isomorphism classes of smooth principal $E$-bundles and  the elements of $H^2(B, \underline{\Gamma})$.  As a consequence we get the following fact: if $M$ is the total space of a principal holomorphic $E$-bundle over $B$ and both $B$ and $M$ are K\"ahler, then $M$ is diffeomorphic to $B \times E$.
\end{rmk}

The first Chern class and other invariants for the Iwasawa bundle and some other  holomorphic principal toric bundles were computed in \cite{Hof}(\S3, \S10) .

\subsection{Iwasawa bundle.}
\label{Iwasawa-ebundle}

\begin{predl}
Let  $\pi \colon I \to T$ be an Iwasawa bundle. Then $c_1(\pi)$ can be represented by a holomorphically symplectic form. In particular $c_1(\pi) \in H^{2,0}(T)$.
\end{predl}
\begin{proof}
Recall that $G$ is the group consisting of matrices of the kind 
$\begin{pmatrix}
1 & z_1 & z_2\\
0& 1 & z_3\\
0&0&1
\end{pmatrix}$ with $z_i \in \C$. This group  acts on $I$ transitively. Let $\mathfrak{g}$ be its Lie algebra. Again denote by $Z(G)$  the center of $G$ and by $G^{ab}$   its abelianization. Their Lie algebras will be denoted as $\mathfrak{z}$ and $\mathfrak{g}^{ab}$ respectively.

 Let $s$ be the cocycle in $Hom(\Lambda^2\C^2, \C)$, which corresponds to the exact sequence of Lie algebras $$ 0 \to \mathfrak{z}  \to \mathfrak{g} \xrightarrow{p} \mathfrak{g}^{ab} \to 0$$. 

Choose any map of $\C$-vector spaces $h \colon \mathfrak{g}^{ab} \to \mathfrak{g}$, which is right inverse to the projection $p$ and  $\mathfrak{g} = \mathfrak{z} \oplus \operatorname{Im} h$ . It induces an Ehresmann $G$-connection $\mathcal{H}$  on $\pi$ and hence affine $G$-connection $\nabla$. 

Let $\sigma$ be the curvature form of $\nabla$, so $c_1(\pi) = [\sigma] \in H^2(T, \C)$. If we restrict  $\sigma$ on the tangent space to the origin in $G^{ab}$, then $$\sigma(X, Y) = [h(X), h(Y)] - h([X,Y]) = s(X,Y)$$

Since both $\nabla$ and $\sigma$ are $G^{ab}$-invariant, $\sigma$ coincides with the form on $G^{ab}$ obtained by left translations from $s$. This means that $\sigma$ is holomorphicaly symplectic (this is a local property and  is preserved by descending to $T$).

\end{proof}

\begin{cor}
Let $\pi \colon I \to T$ be an Iwasawa bundle and $i \colon B \hookrightarrow T$  an embedding of a  curve. Then $c_1(i^*\pi)=0$. Particularly, $\pi^{-1}(B)$ is a K\"ahler surface.
\end{cor}
\begin{proof}
Indeed, $c_1(i^*\pi) = i^*c_1(\pi)$ belongs to $H^{2,0}(C) = 0.$
\end{proof}

\section{Curves in Iwasawa manifold.}
\label{curves}
For this section let $I = G / \Lambda$ be an Iwasawa manifold  and $\pi \colon I \to T$  its Iwasawa bundle which is a holomorphic principal $E$-bundle for elliptic curve  $E = Z(G)/\Lambda_Z$. 
Our aim  is to describe all closed complex curves in $I$.

For the purposes of this section we will need the techinque of Douady spaces.

\subsection{Douady spaces.} Recall that for any complex manifold $X$ there exists its {\it Douady space}: a complex analytic space $\mathcal{D}(X)$ which parametrizes its closed complex subvarieties (\cite{Dou}).  

 The Douady spaces can be viewed as the analytical counterpart  for the Hilbert schemes. For a more detailed introduction into Douady spaces see (\cite{CDGP}, ch. VIII).

In general the Douady spaces are very singular and even non-reduced. However, the Zariski tangent space in a  point $[Z] \in \mathcal{D}(X)$ can be described as the space of global sections of the normal bundle:
$$ T_{Z}\mathcal{D}(X) = H^0(Z, \mathcal{N}_{Z/X}).$$
(see \cite{CDGP}, ch. VIII).

One can also consider the {\it marked Douady space} $\mathcal{D}^{+}(X) \subset \mathcal{D}(X) \times X$ which is the space of pairs $(S, s)$, where $S \subset X$ is a closed subvariety and $s \in S$ is a point. It is equipped with two forgetfull projections 

$$
\xymatrix{
& \mathcal{D}^{+}(X) \ar[dl]_{pr_1} \ar[dr]^{pr_2} & \\
\mathcal{D}(X) &&X
}
$$

For a point $[S] \in \mathcal{D}(X)$  the fibre $pr_2^{-1}([S])$ is isomorphic to $S$. 

Let $C$ be a smooth curve of genus $g$ in complex manifold $X$. Denote by $\mathcal{D}(C, X)$ the  reduction of the connected component of $\mathcal{D}(X)$ which contains $C$. We call it {\it the space of deformations of $C$ in $X$}. Correspondingly, $\mathcal{D}(C, X)$ is the {\it space of marked deformations of $C$ in $X$.}

Recall, that a complex analytic space is said to be {\it of Fujiki class} $\mathcal{C}$ if it is bimeromorphic to a K\"ahler manifold.

The Douady spaces of Fujiki class $\mathcal{C}$ manifolds are described by the so-known Fujiki-Lieberman theorem (\cite{Fuj}, \cite{L}):
\begin{thrm}
Let $X$ be a compact smooth manifold of Fujiki class $\mathcal{C}$ and $\D_0 \subset \D(X)$ be a connected component of Douady space. Then $\D_0$ is compact and its reduction $\D_{0, red}$ is of Fujiki class $\mathcal{C}$ .
\end{thrm}

\subsection{Projections of curves.}  
The aim of this section is to prove the following theorem:

\begin{thrm}
Let $t I$ be  an Iwasawa manifold with the Iwasawa bundle denoted as $\pi \colon I \to T$.  Assume that $C \hookrightarrow I$ is an immersion of smooth curve. Then $\pi(C)$ is either a point or an elliptic curve in $T$.
\end{thrm}

First of all fix the following notation. Assume that $C \subset I$ is a smooth closed curve in an Iwasawa manifold $I$ which is not a fibre of the Iwasawa bundle $\pi \colon I \to T$. Consider the curve $B:=\pi(C) \subset T$ and  let  $\phi_{B} \colon \widetilde{B} \to B$ be its normalization.  Put $I_B:= \pi^{-1}(B)$. The manifold $I_B$ is a (possibly singular) compact complex surface  isomorphic to the total space he restriction of the Iwasawa bundle $\pi_B \colon I_B \to B$. Let also $I_{\widetilde{B}}$ be the total space of $\pi_{\widetilde{B}}:= \phi_B^{*}\pi_B$.  

Notice that by Corollary 3.6, it follows that $c_1(\pi_{\widetilde{B}})=0$. Particularly, by Blanchard theorem (Theorem 3.4) $I_{\widetilde{B}}$ is a K\"ahler surface.

First of all observe that if $B$ is not an elliptic curve, the genus of  $C$ is greater than one. 

Next, we show that without loss of generality  we may assume that $\pi \colon C \to B$ is of degree $1$. 

\begin{predl}
Assume that there exists a smooth curve $C$  in an Iwasawa manifold $\pi \colon I \to T$, such that $B =\pi(C$)  is a curve of geometric genus $g(B) > 1$. Then there exists another Iwasawa manifold $I'$ with Iwasawa bundle $\pi' \colon I' \to T$ and a curve $C' \subset I'$, such that $\pi'(C') = B$  and $\pi' \colon C' \to B$ is of degree one.
\end{predl}
\begin{proof}
Let $k$ be the degree of $\pi \colon C \to B$. Since $C$ is smooth, the projection $C \to B$ factors through the normalization $\phi_B \colon \widetilde{B} \to B$. The curve  $C$ defines a meromorphic multisection of $\pi_B$, which provides a well-defined $k$-fold multisection of the principal elliptic bundle $\pi_{\widetilde{B}} = \phi_B^* \pi_B$. Let $\pi' := \pi^{\times_E k}$ be the holomorphic principal elliptic bundle on $T$ which is obtained by a multiplication of $[\pi] \in H^1(T, \mathcal{E}_T)$ by the integer $k$. Denote  by $I'$  the total space of $\pi'$. 

The bundle $\pi_{\widetilde{B}}$ admits a $k$-fold multisection, so by Proposition 3.3 it belongs to the $k$-torsion subgroup of $H^1(\widetilde{B}, \mathcal{E}_{\widetilde{B}})$. The natural map $H^1(T, \mathcal{E}_T) \xrightarrow{\phi_B^*i^*} H^1(\widetilde{B}, \mathcal{E}_{\widetilde{B}})$ obtained as the composition of restriction on $B$ and pull-back $\phi^*_B$ is a group homomorphism. Hence $$\pi'_{\widetilde B}:= \phi_B^*i^*(\pi') = k \cdot (\phi_B^*i*(\pi) = k \cdot \pi'_{\widetilde{B}}$$ is trivial. 

The principal bundle $\pi'_B \colon I'_B \to B$ therefore admits a meromorphic section, which is well-defined on the smooth locus of $B$. The closure of its image in $I'_B$ is the required curve $C'$. It is clear, that $C'$ is smooth, since it is isomorphic to $\widetilde{B}$.
\end{proof}

Recall that $\mathcal{D}(I)$ is the Douady space of Iwasawa manifold and $\mathcal{D}^{+}(I) \subset \mathcal{D}(I) \times I$ is the marked Douady space. We have natural projections

$$
\xymatrix{
& \mathcal{D}^{+}(I) \ar[dl]_{pr_1} \ar[dr]^{pr_2} & \\
\mathcal{D}(I) &&I
}
$$

\begin{predl}
Let $\tau_v \colon T \to T$ be the translation by an element $v \in T$. Then for any  $B' = \tau_v(B)$ there exists canonical isomorphism $\psi_v \colon H^{\bullet}(I_{B'}, \Z) \simeq H^{\bullet}(I_B, \Z)$. 
\end{predl}
\begin{proof}
The group $G$ of complex $3 \times 3$ unipotent matrices acts on $I$ holomorphically and transitively. Abusing the language we denote this action by the same letter $\tau$.  The canonical projection  $p \colon G \to T = G^{ab}/\Lambda^{ab}$  together with $\pi \colon I \to T$ forms a morphism of homogenous manifolds. That is $\tau_g \circ \pi = \pi \circ \tau_{p(g)}$ for any $g \in G$.

Choose any $g \in G$, such that $p(g) = v$. Then $\tau_g$ defines an isomorphism $I_B \simeq I_{B'}$. Put $\psi_v:=(\tau_g)_*$. In spite of the fact, that the isomorphism $\tau_g$ depends on the choice of $g$, the induced isomorphism on cohomology doesn't. Indeed, if $g_1$ and $g_2$ are two different liftings of $v$, then $g_1g_2^{-1}$ belongs to the centre of $G$. Thus, $\tau_{g_1} \circ \tau_{g_2}^{-1} \colon I_B \to I_B$ is just the action by an element of $E$ on $I_B$ viewed as the total space of the principal $E$-bundle $\pi_B$. Therefore it is homotopically trivial and $\psi_v$ is well-defined.

\end{proof}

\begin{lemma}
Assume that $C \subset I$ is  a smooth immersed curve in Iwasawa manifold, such that $\pi \colon C \to B$ is of degree $1$ and $\widetilde{B}$ is of genus $g \ge 2$. Then there exists a compact subvariety $K \subset \mathcal{D}(I)$ such that $pr_2(pr_1^{-1}(K)) = I$. 
\end{lemma}
\begin{proof}
The proof is given in three steps. The first step is to construct $K$. The second step is to prove that the constructed space is compact and the third one is to verify that its pre-image $K^{+}:= pr_1^{-1}(K) \subset \D^{+}(I)$ projects onto $I$ surjectively.

The curve $C$ is smooth and $\pi \colon C \to B$ is of degree $1$, so $C \simeq \widetilde{B}$. The curve $C$ gives a trivialization of $\pi_{\widetilde{B}}$ and hence an isomorphism $I_{\widetilde{B}} \simeq C \times E.$ Fix the cohomology class $[C] \in H^2(I_B, \Z)$.

Let $K$ be the set of (possibly singular) curves $C' \subset I$, such that $\pi(C') = \tau_v(\pi(C))$ for some $v \in T$ and $\psi_v([C']) = [C]$. It is clear that $K \subset \mathcal{D}(I)$ is a subvariety.

\begin{predl}
In the assumptions of Lemma 4.5 the space $K$ is compact.
\end{predl}
\begin{proof}
The projection $\pi \colon I \to T$ defines a morphism of Douady spaces $\pi_* \colon \mathcal{D}(I) \to \mathcal{D}(T)$. 

Let  $T_B$ is the orbit of $B \in \mathcal{D}(T)$ under the action of $T$ by translations. The morphism $\pi_*$  maps  $K$ to $T_B$ surjectively. To see this for any $v \in T$ choose a lifting $g \in G$ (that is, $p(g) = v$). Then $\pi_*^{-1}(\{\tau_v(B)\}) \bigcap K$ contains $\tau_g(C)$ and hence non-empty.

Fix any $B' \in T_B$ and consider the fibre $\pi_*^{-1}(B')$. It is a connected component of the Douady space of $I_{B'} = \pi^{-1}(B')$. By Corollary 3.6 $I_{B'}$ admits K\"ahler metric, and hence by the Fujiki - Lieberman theorem (Theorem 4.1) it is compact.

The morphism $\pi_* \colon K \to T_B$ has compact fibres, hence it is proper. The complex analytic space $T_B$ is also proper. It follows that $K$ is also a proper analytic space, hence it is compact. We proved the Proposition 4.6.
\end{proof}

Now we return to the proof of Lemma 4.5. It is left to verify that $pr_2 \colon K^{+} \to I$ is surjective. Let $x$ be an arbitrary point in $I$. Choose any point $x_0 \in C$. Let $v:= \pi(x) - \pi(x_0)$. As it was discussed before, there exists a curve $C_1 \in K$, such that $\pi(C_1) = \tau_v(C)$.There exists a point $x_1 \in C_1$, such that $x_1$ and $x$ belong to the same fibre of the Iwasawa bundle. Their difference $x_1 - x =: t$ is a well-defined element of $E$ and the image of $C_1$ under the translation by $t$ along the fibres of Iwasawa bundle is a curve $C' \in K$ which passes through $x$.

\end{proof}

Now we explain how the last lemma leads to the contradiction in the proof of the Theorem 4.2

\begin{proof}[Proof of Theorem 4.2]
Assume the opposite. By Proposition 4.3 it is sufficient to come to the contradiction in the case when $\pi \colon C \to B$ is of degree $1$.  Thus we may consider the analytic spaces $K$ and $K^{+}$ as in the Lemma 4.5. Let $\phi \colon W \to K$ and $\phi^{+} \colon W^{+} \to K^{+}$ be their desingularizations.  Therefore we obtain the following diagram:
$$
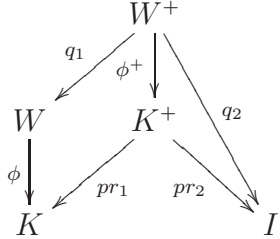
\xymatrix{
&W^{+} \ar[dl]_{q_1} \ar[d]_{\phi^{+}}  \ar[ddr]^{q_2}\\
W \ar[d]_{\phi}& K^{+} \ar[dl]^{pr_1} \ar[dr]_{pr_2} & \\
K &&I
}
$$
Here $q_2 := pr_2 \circ \phi^{+}$ and $q_1$ exists by the universal property of $W$.

Let $\alpha, \beta$ and $\gamma$ be  holomorphic $1$-forms on $I$ as in Theorem 2.1. Recall that $d\alpha = d\beta = 0$ and $d\gamma = -\alpha \wedge \beta$. Put $\tau : = \alpha \wedge \beta \wedge \overline{\alpha} \wedge \overline{\beta}$ and $$\omega := (\alpha \wedge \overline{\alpha} + \beta \wedge \overline{\beta}) \wedge \gamma \wedge \overline{\gamma}.$$

The form $\omega$ is a strictly positive closed $(2,2)$ form on Iwasawa manifold (the Iwasawa manifold is {\it balanced}: a complex $n$-dimensional manifold is said to be balanced if it admits a closed positive $(n-1, n-1)$-manifold). The form $\tau$ is exact and positive.

Then $(q_1)_*(q_2)^*\omega$ is a positive closed $(1,1)$-current on $W$, which is strictly positive almost everywhere. Moreover since the cohomology class of $\omega$ is integral, the cohomology class of $(q_1)_*(q_2)^*\omega$ is integral as well. Thus, by a theorem of Popovici (\cite{P}), $W$ is Moishezon (that is bimeromorphic to a  projective manifold).

However, $(q_1)_*(q_2)^*\tau$ is a positive exact $(1,1)$-current on $W$. It is a standard fact that a Moishezon manifold can never carry a positive exact current. To see this one can take a push-forward of  such a current to a K\"ahler maifold and notice that its pairing with the K\"ahler form should  be strictly positive and vanish at the same time.

\end{proof}

\subsection{Curves  in Iwasawa manifold.}
\begin {predl}
Let $B$ be an elliptic curve in $T$. Then the restriction of $\pi$ on $B$ is trivial.
\end{predl}
\begin{proof}
Without loss of generality we might assume that $B$ passes through the origin.

The exact sequence of complex Lie groups $$1 \to \C \to G \to \C^2 \to 1$$ splits over any line $L \subset \C^2$. This splitting defines a trivialization of $\pi$ over any 1-dimensional complex subgroup in $T$.
\end{proof}

\begin{cor}
Any curve $C \subset I$ is contained in some holomorphic subtorus. 
\end{cor}
\begin{proof}
If $C$ is in the fibre of $\pi$, then it is isomorphic to $E$ as a complex manifold. Otherwise $\pi(C) = E'$ is necessarily an elliptic curve in $T$  and $C \subset \pi^{-1}(E') = E' \times E$.

\end{proof} 

In the Appendix we prove that such a torus always has maximal Picard number. Indeed both $E$ and $E'$ necessarily carry (the same) complex multiplication.

\section{Surfaces in Iwasawa manifold}
\label{surfaces}
Now let $S \subset I$ be a complex surface. Since $\pi$ does not admit a multisection, $\pi(S) \neq T$. This means that $\pi(S)$ is a curve. One easily gets the following classification theorem:

\begin{thrm}
Let $S \subset I$ be a complex surface. Then there are two possibilities : 
\begin{enumerate}
\item $S$ is an abelian surface isomorphic to a product of two elliptic curves.
\item $S$ is a  K\"ahler  non-projective  isotrivial elliptic surface of Kodaira dimension 1. The surface $S$ is diffeomorphic to $C \times E$, where $E$ is an elliptic curve and $C$ is a curve of genus $g \ge 2$.
\end{enumerate}
\end{thrm}
\begin{proof}
The first case takes place when $\pi(S) = B$ is of genus one. Then $S$ is the total space of $\pi_B$. By Proposition 4.7 we get $S \simeq B \times E$.

The second case takes place when genus of $C$ is greater or equal then 2. Then $S$ is a total space of principal $E$-bundle $\pi_C$ over $C$ which is the restriction of $\pi$ to $C$. The surface $S$ is K\"ahler by Corollary 3.6 . Consequently,  $S$ is diffeomorphic to $C \times E$ (see the Remark on the page 6) . However, $\pi_C$ admits no holomorphic multisections (a multisection would provide a counter-example to  Lemma 4.8).  We claim that $S$ cannot be projective. 

Indeed, suppose that $S$ is embeded into projective space. Then by Bertini's theorem one could find an irreducible hyperplane section of $S$ transversal to the fibres of $\pi_C$. This would deliver a finite multisection of $\pi_C$.

\end{proof}

It is an interesting consequence that a posteriori all complex surfaces in Iwasawa manifolds are K\"ahler.

\begin{cor}
Let $f\colon X \to I$ be a holomorphic map from a smooth projective variety to an Iwasawa manifold. Then $f$ factors through an abelian variety.
\end{cor}
\begin{proof}
It is a standard fact that image of projective variety under holomorphic map is projective. The claim follows from Corollary 4.8 and Theorem 5.1
\end{proof}

\begin{rmk} In \cite{FGV} Fino, Grantcharov and Verbitsky proved that any holomorphic morphism from a complex nilmanifold to a projective variety factors through an abelian variety. Our results might be considered as a particular case of the  (conjectural) dual statement. It is an open question if it is true or not in the general case.

\end{rmk}

\section{Appendix: Iwasawa manifolds and complex multiplication.}
\label{CM}
One might be interested in the following question: which tori $T$ and  elliptic curves $E$ can be obtained as the base and the fibre of an Iwasawa bundle. 

Let $\mathfrak{g}$ be the Lie algebra of  $G$. Let $\mathfrak{z}$ and  $\mathfrak{g}^{ab}$ be its centre and abelianization repsectively. Since the group $G$ is nilpotent, the exponential map $\exp \colon \mathfrak{g} \to G$ is affine. The cocompact lattices  in $G$ are in one to one correspondence with lattices of full rank in $\mathfrak{g}$, which are closed under the Lie bracket (\cite{Mal}, \cite{Rol}). 

So let $\Lambda$ be  a cocompact lattice in $G$ and $\log(\Lambda)$ its formal logarithm which is a Lie subring in $\mathfrak{g}$. Consider the lattices $\Gamma := \log(\Lambda) \bigcap \mathfrak{z}$ and $\Delta := \log(\Lambda) /\Gamma $. The above-mentioned Mal'\v{c}ev theorem implies these are lattices of the full rank in $\mathfrak{z}= \C$ and $\mathfrak{g}^{ab} = \C^2$ respectively. 

For any $\C$-linear section of the natural projection $h \colon \mathfrak{g}^{ab} \to \mathfrak{g}$ the cocycle map $$\Lambda^2 \mathfrak{g}^{ab} \to \mathfrak{z}, \  v_1 \wedge v_2 \mapsto [h(v_1), h(v_2)]$$ defines a $\C$-linear non-degenerate symplectic form $q \colon \Lambda^2\C^2 \to \C$, such that $$q(\Lambda^2\Delta) \subset \Gamma$$ (see Proposition 3.6).

Conversely, if for a fixed non-degenerate form $q \in \Hom_{\C}(\Lambda^2\C^2, \C)$ and for some lattices $\Delta \subset \C^2$ and $\Gamma \subset \C$ the condition $$q(\Lambda^2\Delta) \subset \Gamma$$ is satisfied, then there exists a  cocompact lattice $\Lambda$ in $G$,  such that $\log(\Lambda) \bigcap \mathfrak{z} = \Gamma$ and $\log(\Lambda) / (\log(\Lambda) \bigcap \mathfrak{z}) = \Delta$. It is given by 

$$\Lambda : = \left. \left\{ \begin{pmatrix} 1 & \delta_1 & \frac{\gamma}{2} \\ 0 & 1 & \delta_2 \\ 0 & 0 & 1 \end{pmatrix} \right| \delta_1, \delta_2 \in \Delta, \gamma \in \Gamma \right\} .$$ 

For the corresponding Iwasawa manifold $I = G/\Lambda$ the base of the Iwasawa bundle $T = G^{ab}/\exp(\Lambda)^{ab}$ is isomorphic to $\C^2/\Delta$ and its fibre $E = Z(G)/\exp(\Lambda \bigcap \mathfrak{z})$ is isogenous to the elliptic curve $\C/\Gamma$.

It turns out that the condition $q(\Lambda^2\Delta) \subset \Gamma$ is rather strong. It implies that  the curve $E$ always carries a complex multiplication and the torus $T$ is isogenous to a product of two elliptic curves with the same complex multiplication. 

Recall that for a general complex torus  of  complex dimension $g$   its endomorphism ring is isomorphic to $\Z$. A complex torus $A$ of dimension $g$  is said to carry {\it complex multiplication} (or to be {\it of CM-type}) if its rational endomorphism algebra has dimension $\dim_{\Q}\operatorname{End}(A) \o \Q = 2g^2$. 

For elliptic curves this notion is especially well-known:

\begin{predl}
Let $E$ be an elliptic curve. Then the following conditions are equivalent:
\begin{enumerate}
\item There exists a holomorphic automorphism $\tau \colon E \to E$, which fixes the origin and is not equal to $\pm \mathrm{Id}$;
\item $\operatorname{End}(E) \o \Q = K$ for some  imaginary quadratic number field $K$;
\item $E = \C/\sigma(O)$, where $O$ is an order in an imaginary quadratic number field $K$ and $\sigma \colon K \hookrightarrow \C$ is an embedding;
\item $E$ is an elliptic curve with complex multiplication.
\end{enumerate}
\end{predl}
\begin{proof} This is a standard fact, which can be found, e.g.,  in (\cite{Sil}, ch.V)
\end{proof}

The classical sources on the theory of complex multiplication on elliptic curves and abelian varieties are \cite{Mum} and \cite{Sil}.

In a more general situation a close statement about complex multiplication on tori in nilmanifolds was proved by J. Winkelmann in \cite[ch. 9]{W}

\begin{thrm}[Winkelmann]
\label{Wink}
Let $G$ be a nilpotent complex  Lie group which is irreducible(i.e. is not a product of two other complex nilpotent groups) and not abelian.Consider a nilmanifold $N = G/\Lambda$. 

Then $N$  decomposes into a sequence of holomorphic principal toric bundles $$N=N_0 \xrightarrow{T_0} N_1 \xrightarrow{T_1} N_2 \xrightarrow{T_2} \ldots \xrightarrow{T_{s-2}} N_{s-2} \xrightarrow{T_{s-1}} N_{s-1} = T_s.$$
Moreover for any $j < s$ the torus $T_j$ is isogenous to a product of simple tori with complex multiplication. The same holds for $T_s$ if and only if it is algebraic.
\end{thrm}

Winkelmann's proof is rather sophisticated and tells nothing about the geometry of the basic torus $T_s$ when it is not algebraic. We are going to show that in the Iwasawa case the basic torus $T_s = T$ has the maximal Picard number, which  not only yields its algebraicity, but gives the result of  the Theorem 6.2 for  Iwasawa manifolds.

\begin{lemma}
The base $T$ of an Iwasawa bundle $\pi$ has maximal Picard number.
\end{lemma}

\begin{proof}
Consider the space $H^{2,0 + 0,2}_{\Q}:= H^2(T, \Q) \bigcap (H^{2,0}(T) \oplus H^{0,2}(T))$. It is enough to prove that $\dim_{\Q} H^{2,0+0,2}_{\Q} = 2$.

Assume that $E = \C/\Gamma$ is the structure group of $\pi$ and consider $\Gamma_{\Q} := \Gamma \o \Q$ viewed as a $\Q$-vector subspace of $\C$.

Recall that the image of the cocycle $q \colon \Lambda^2\g^{ab} \to \mathfrak{z}$ generates the whole center $\mathfrak{z}$ of the Lie algebra $\g$. The cohomology class of the induced left-invariant form on $T$ is equal to $c_1(\pi)$ and generates $H^{2,0}(T)$.  Since $q$ is defined over $\Q$ we get that $$\dim_{\Q} H^2(T, \Gamma_{\Q}) \bigcap H^{2,0}(T) = \dim_{\Q} \log(\Gamma) \bigcap \mathfrak{z} = 2.$$ If $\gamma_1$ and $\gamma_2$ are two linearly independent classes in $H^2(T, \Gamma_{\Q}) \bigcap H^{2,0}(T)$, then $\gamma_1 + \overline{\gamma_1}$ and $\gamma_2 + \overline{\gamma_2}$ are two linearly independent rational classes in $H^{2,0 + 0,2}_{\Q}$.
\end{proof}

Next we use the following  lemma:
\begin{lemma}
Let $A$ be a complex torus of dimension $g \ge 2$ and $\rho(A)$ its Picard number. Then the following conditions are equivalent:
\begin{enumerate}
\item $A$ has the maximal Picard number, i.e. $\rho(A) = h^{1,1}(A) = g^2$;
\item$ H^{1,1}(A)$ is defined over $\Q$;
\item $H^{2,0}(A) \oplus H^{0,2}(A)$ is defined over $\Q$.
\item $\dim_{\Q} \operatorname{End}(A) \o \Q = 2g^2$;
\item $A$ is isogenous to the $g$-th power of $E$, where $E$ is an elliptic curve with complex multiplication.
\item $A$ is isomorphic to $E_1 \times \ldots \times E_s$, where all $E_s$ are mutually isogenous elliptic curves with complex multiplication.\label{prod}
\end{enumerate}
\end{lemma}
\begin{proof} See e.g. \cite{Beau}
\end{proof}

Hence if $T = \C^2/\Delta$ is  the base of Iwasawa bundle $\pi$, then it is isomorphic to the product of two elliptic curves $E'$ and $E''$ with the same complex multiplication. Particularly, this means that the $\Q$-vector space $\Delta \o \Q \subset \C^2$ can be decomposed into a direct sum $\sigma(K) \oplus \sigma(K)$, where $K$ is an imaginary quadratic number field and $\sigma$  is one of its two conjugated embeddings into $\C$.

Let $E = \C/\Gamma$ be the structure group of the Iwasawa bundle over $T$. As it was discussed in the beginning of this section, there is a holomorphically symplectic form $q \colon \Lambda^2\C \to \C$ such that $$q(\Lambda^2\Delta \o \Q)  = \Gamma \o \Q.$$ Clearly, this implies that $\Gamma \subset \sigma(K)$ and $E$ carries the same complex multiplication as $E'$ and $E''$.

\begin{cor}
Let $\pi \colon I \to T$ be an Iwasawa bundle with the structure group $E$. Then $T$ is isogenous to a product of two elliptic curves $E' \times E''$. All the three curves $E, E'$ and $E''$  carry the same complex multiplication.
\end{cor}

\begin{predl}
Let $A$ be a complex torus of maximal Picard number and $B \subset A$ be an elliptic curve. Then $B$ has complex multiplication.
\end{predl}
\begin{proof}
For the first step notice that $B$ carries  complex multiplication if and only if it is isogenous to a certain curve $B'$ with  complex multiplication.
Due to the property \ref{prod} from  Lemma 6.4  the torus $A$ is isomorphic to $E_1 \times \ldots \times E_s$, where $E_j$ are  isogenous elliptic curves with the same complex multiplication.

Let $pr_j \colon A \to E_j$ be the natural projections. At least one of this projections  defines an isogeny $pr_j \colon B \to E_j$. Hence $B$ carries the same complex multiplication as $E_j$.
\end{proof}

\begin{predl}
Let $A \subset I$ be a holomorphic torus. Then $A$ carries a  complex multiplication. 
\end{predl}
\begin{proof}
Suppose that  $\dim_{\C}A = 1$, i.e. $A$ is an elliptic curve.  As follows from Lemma 4.7 either $A$ is a fibre of $\pi$ and hence $A \simeq E$, or $\pi(A) = B$ is an isogenous to $A$  curve in $T$.  Due to Proposition 6.6 and Lemma 6.3 $B$ is an elliptic curve with complex multiplication.

Now assume that $\dim_{\C}A = 2$. Then $\pi(A) = B$ is an elliptic curve. Moreover, as it is proved in Proposition 4.11, $A \simeq B \times E$. Thus, Proposition 6.6 and Lemma 6.4 imply that the Picard number of $A$ equals 4.
\end{proof}

\end{document}